\documentclass[authoryear]{article}

\usepackage{amsmath, amssymb, amsthm}

\theoremstyle{plain}
\newtheorem{thm}{Theorem}[section]
\newtheorem{lem}[thm]{Lemma}
\newtheorem{prop}[thm]{Proposition}

\theoremstyle{definition}
\newtheorem{rem}{Remark}[section]
\newtheorem{exam}{Example}[section]

\usepackage{natbib}

\usepackage{type1cm}        
\usepackage{makeidx}         
\usepackage{graphicx}        
\usepackage{multicol}        
\usepackage[bottom]{footmisc}

\usepackage{newtxtext}       %
\usepackage{newtxmath} 

\makeindex             

\numberwithin{equation}{section}

\begin{document}
\title{Brown's loss function and estimation of the scale parameter and the entropy in  scale mixture models with unknown location}
%
\author{Stavros Kourouklis\\[1ex]
	\small{Department of Mathematics, University of Patras, 26504 Patras, Greece}\\ 
		\small{e-mail: stavros@math.upatras.gr}
}

\date{}

\maketitle

\begin{abstract}
	This paper deals with the problems of decision theoretic estimation of the scale parameter, $\sigma$, and the entropy (Shannon or Renyi) in general scale mixture models with unknown location. Employing Kubokawa's approach, under convexity of  the loss $L(t)$ either in $t$ or in $lnt$, improved  estimators of $\sigma$ are derived. If $L(t)$ satisfies an additional condition, the problem  reduces to estimating  $\sigma$ in the degenerate (scale mixture) model. As the entropy  is a linear translation of $ln\sigma$, its estimation under a  loss $L_0(u)$ amounts to estimating $\sigma$ under $L(t)=L_0(lnt)$, with the quadratic loss $L_0(u)=u^2$ and Brown's loss $L(t)=(lnt)^2$ being a typical case. Thus, in the case that $L_0(u)$ is convex improved estimators of the entropy are directly obtained from respective estimators of $\sigma$. The results are applied to scale mixtures of normal and exponential distributions, including the multivariate $t$ and multivariate Lomax distributions.
\end{abstract}
{\footnotesize\noindent{\bf Keywords:} decision theory, improved estimation of a scale parameter, Stein-type estimator, Brewster and Zidek-type estimator, Strawderman-type estimator, Kubokawa's approach, scale mixture model}


%
\numberwithin{equation}{section}
\section{Introduction}\label{sect introduction}

The problem of decision theoretic point estimation of a scale parameter in the presence of a nuisance parameter has  a long history tracing back to the seminal work of \cite{Stein1964}. In estimating the variance of a normal distribution with unknown mean, \cite{Stein1964} proved the inadmissibility of the usual (best location scale equivariant) estimator by constructing a (scale equivariant) estimator with smaller mean squared error.

\cite{Brown1968} extended \cite{Stein1964}'s result to general location-scale families of distributions and general (bowled-shaped) loss functions. Furthermore, \cite{Brown1968}  discussed losses which are functions of $|lnt|$ defined  on $[0, +\infty)$,  "as possibly reasonable choices since a scale parameter is being estimated" (the complete reasoning is given on page 37 of his paper). \cite{Brown1968} pointed out that "A mathematically convenient loss function of this type is given by" $(lnt)^2$, and studied properties of this loss in relation to the quadratic loss and the entropy loss. In the literature, $(lnt)^2$ is known as  Brown's loss.

\cite{BrewsterZidek1974}, \cite{Strawderman1974}, and \cite{Kubokawa1994} are further landmark works and classical tools for deriving improved estimators of a scale parameter.

 In this paper we consider the following mixture model. Given $\tau >0 $,  $S$ and  $X$ are independent statistics such that $ S/\big(\rho(\tau)\sigma\big)$ and $ X/\big(\rho(\tau)\sigma\big)$ have densities
\begin{equation}\label{1.1}
g(v)I(v>0) \quad\text{and}\quad h(x;\lambda)I(x>\kappa(\lambda))
\end{equation}
respectively, where $\rho(\tau)>0$ is a known function,  $\lambda=\lambda(\mu,\sigma, \tau)$ is a function of
$\mu$, $\sigma$, $\tau$, 
$\kappa(\lambda)$ is a real function of $\lambda$, and $\mu$, $\sigma$ unknown parameters with $\sigma>0$. The mixing  parameter $\tau$ is assumed to have a  distribution function $G(\cdot)$. In the sequel, for the sake of notational simplicity, $\tau$ will denote either a random variable with distribution function $G(\cdot)$ or a fixed value of this random variable.  Also,  without loss of generality we may and will take $\rho(1)=1$. When $\rho(\tau)=1/\tau$, \eqref{1.1} is the mixture model considered in \cite{PetropoulosKourouklis2005}.
 
 In the degenerate  case, i.e.,   $\tau=1$ wp1 with respect to $G(\cdot)$, \eqref{1.1} becomes \cite{Kubokawa1994}'s framework of independent data $(S,X)$
 that  encompasses after reduction by sufficiency, as typical cases, independent data from a normal distribution (i.i.d. or regression schemes) or an  exponential distribution with various sampling schemes, such as i.i.d., Type-II censoring, progressive Type-II censoring, adaptive progressive Type-II censoring, and record values.
 
 In the general non-degenerate case, \eqref{1.1} includes scale mixtures of normal (SMN) distributions and scale mixtures of exponential (SME) distributions. SMN distributions form an attractive subclass of the class of spherically symmetric distributions and include several interesting members with heavier tails than the normal distribution. In particular, members of SMN distributions are the multivariate $t$, the multivariate double exponential, the slash, and the multivariate Pearson type VII distributions (with $\rho(\tau)=1/\tau$, except for the slash $\rho(\tau)=\tau^{1/q}$). Recent developments of SMN distributions are reviewed in \cite{SMD}. The model of SME distributions was introduced by \cite{LandS} in the context of reliability studies of  a system composed of several components which  share a common operating environment. It  includes as  members the multivariate Lomax and the multivariate modified Lomax distributions ($\rho(\tau)=1/\tau$),   which have heavier tails than the exponential distribution.

 The aim of this work is to contribute to the decision theoretic estimation of the scale parameter $\sigma$ and the entropy, Shannon or Renyi, of the model \eqref{1.1}.

 We now describe our results and position them within the relevant research. In Theorem \ref{thm2.1}, we give sufficient \cite{Kubokawa1994}-type  conditions for an estimator $\delta(S,X)$ to improve on the usual estimator $\delta_0$ of $\sigma$, which is a function, only, of $S$. The loss $L(t)$ is assumed to be either strictly convex  in $t \in (0,\infty)$ or strictly convex in $lnt$, hereafter (strictly) convex in log. In the sequel, the term convex will mean strictly convex.  Theorem \ref{thm2.1} may be viewed as an extension of \cite{Kubokawa1994}'s unified approach on point estimation of a scale parameter (specifically, his Theorem 2.2), from independent data  to the dependent scale mixture data in 
 \eqref{1.1}.  It is stated in \cite{Kubokawa1994} that "... the results can be extended to the case that ... are not independent by replacing ... with suitable assumptions", so Theorem  \ref{thm2.1} presents such a (general and useful) case. Whether the convexity  may be relaxed to (strictly) bowl-shaped as in  \cite{Kubokawa1994} is an issue of interest, see Remark \ref{Remark2.1}. Note that $L(t)$  convex in log is equivalent to $L(t)$ having the form $L(t)=L_0(lnt)$, where $L_0(u)$ is convex in $u \in (-\infty, +\infty)$, for instance, $L(t)=(lnt)^2$ with $L_0(u)=u^2$.

Under an additional condition on the loss due to \cite{Iliop}, namely,
\begin{equation}\label{Ilio}
L'(xy)=\lambda_1(x)L'(y)+\lambda_2(y)L'(x), \, \, \,\, \, x>0,\, y>0
\end{equation}
for some functions $\lambda_1$ and  $\lambda_2$, in Theorem \ref{reddegen} it is shown  that any improved estimator of $\sigma$ obeying the conditions of Theorem \ref{thm2.1} is just a  constant multiple of a corresponding improved estimator of $\sigma$ in the degenerate model ($\tau=1$ wp1)  obeying  \cite{Kubokawa1994}'s conditions (i.e. his Theorem 2.2 or Theorem \ref{degenerate} in Section 2). In this way, we immediately obtain \cite{Stein1964}-type (S-type) and  \cite{BrewsterZidek1974} -type (BZ-type) estimators of $\sigma$ from their degenerate case counterparts which are already available in the literature. Convexity in log and condition \eqref{Ilio} are satisfied by several typical loss functions, including the entropy loss and \cite{Brown1968}'s loss $(lnt)^2$, see Table 1 in  \cite{Bobotas2012}.

We mention  that the results of  Section \ref{sect scale mixtures} unify and generalize main findings in  \cite{PetropoulosKourouklis2005}, see Remark \ref{Remark2.3}. Section \ref{sect scale mixtures} is of independent interest and  can stand by itself, yet it plays a crucial supportive role for Section \ref{sect entropy}.

   Section \ref{sect entropy} deals with the problem of estimation of \cite{Shannon1948} entropy for the mixture model in \eqref{1.1}. For a probability measure (probability distribution) with absolutely continuous density $p$ with respect to a $\sigma$-finite measure $\nu$,  Shannon entropy of $p$  is defined as $ e(p)=- \int (lnp)p d\nu$.
   Originally introduced  by \cite{Shannon1948} as a measure of uncertainty or information in the context of a communication system, it has long been recognized as relevant, potentially, to any system operating in a random environment. In this way, Shannon entropy (hereafter entropy) has opened new paths of research and applications in many and diverse disciplines.
   
   By invariance under location transformations, the entropy of the model in \eqref{1.1} is represented by the entropy of the probability measure (the distribution ) of $S$, which is  $e(\sigma)= ln\sigma$ +e(1). Here,  as the notation suggests, $e(1)$ is the entropy when $\sigma=1$, a known constant for given $\rho(\tau)$ and $G(\cdot)$. In Proposition \ref{prop 3.1}, we show in an elementary way that, given a loss $L_0(u)$ for $u \in (-\infty, +\infty)$, if $\delta$ is an improved estimator of $\sigma$ under $L(t)=L_0(lnt)$, so is $\zeta=ln\delta$ for $ln\sigma$ and $\zeta +e(1)$ for the entropy $e(\sigma)$ under $L_0(u)$. Besides, the method of construction of $\delta$ is immaterial.
     As a consequence, in the case that $L_0(u)$ is convex the results of Section \ref{sect scale mixtures} for estimating $\sigma$  lend themselves to directly estimating $ln\sigma$ and $e(\sigma)$, see Theorem \ref{thm4.2}. For instance, S-type and BZ-type estimators of $\sigma$ in SMN distributions or SME distributions under \cite{Brown1968}'s loss $L(t)=(lnt)^2$ derived in  Section \ref{sect scale mixtures} directly yield S-type and BZ-type estimators of $ln\sigma$ and  $e(\sigma)$ under the quadratic loss $L_0(u)=u^2$. In the special case of estimating the entropy of the degenerate model, convexity is not needed and  it is sufficient that  $L_0(u)$ be strictly bowl-shaped.
      
       Besides, using \cite{PalLing1996}'s \cite{Strawderman1974}-type (ST-type) estimator for the normal variance $\sigma$ under $L(t)=(lnt)^2$, $\delta_{ST}$ say, then Proposition \ref{prop 3.1} produces the ST-type estimator $ln\delta_{ST}$ of $ln\sigma$ under $L_0(u)=u^2$ and a corresponding one for the entropy of the normal distribution. Similarly, ST-type estimators of $\sigma$  under the entropy loss $L(t)=t-lnt-1$ given by \cite{PalLing1995} and \cite{MaruyamaStrawderman2006} yield ST-type estimators of $ln\sigma$ and the entropy under the Linex loss $L_0(u)=e^u-u-1$. We mention that we have constructed ST-type estimators of the scale parameter and the entropy of an exponential distribution  under Brown's loss and quadratic loss, respectively, but we   have not included them here. Also, ST-type estimators of the scale parameter of an exponential distribution  under the entropy loss have been obtained by \cite{Bobotas2011}, thus yielding ST-type estimators of the entropy under the Linex loss. 
   
   In  the literature there are available S-type and BZ-type estimators of $ln\sigma$ and the entropy for the exponential distribution and for  scale mixtures of exponential distributions, see \cite{KK}, \cite{Patra19}, \cite{Patra20}, \cite{Petroetal}. In those models (which are covered by \eqref{1.1}), estimation of $ln\sigma$ is treated as a problem unlinked  to that of estimating $\sigma$, thus requiring extra (unnecessary) arguments  in order to obtain the improved estimators.  In this work, estimation of $ln\sigma$ results as  a by-product of estimating $\sigma$. The advantage of this approach is evident, for instance, in the degenerate case ($\tau=1 wp1)$ improved estimators of $\sigma$ already available in the literature directly yield respective improved estimators for $ln\sigma$ and the entropy.

     We note that all the results of Section \ref{sect entropy} apply as well  to the estimation of Renyi entropy $\varepsilon_r(p)=(1/(1-\alpha)) ln \int p^{\alpha} d\nu$, which in the case of (the distribution of) $S$  in \eqref{1.1}  is expressed as  $\varepsilon_r(\sigma) = ln\sigma$	 + $\varepsilon_r (1)$,
      where  the constant $\varepsilon_r (1)$ is Renyi entropy when $\sigma=1$.

     In summary, this paper sheds new light on  the improved estimation of the scale parameter in scale mixture models and presents a systematic treatment of  the decision theoretic estimation of the entropy of those models.  Brown's loss $(lnt)^2$ plays an instrumental role in linking the two problems.

\section{Estimation of the scale parameter in scale mixture models}\label{sect scale mixtures}

In this section, we  consider the mixture model in \eqref{1.1} with fixed and known (completely specified) $G(\cdot)$. Based on $(S,X)$, we want to estimate $\sigma$ by an estimator $\delta$ incurring loss $L(\delta/\sigma)$. We assume that $L(t)$ is  convex on $(0,\infty)$ or  convex in log on $(0,\infty)$,with minimum $0$ at $t=1$. That is, $L'(t)$ or $tL'(t)$ is strictly increasing (referred to as increasing from here on), with $L(1)=0$ and $L'(1)=0$. Recall that $L(t)$  convex in log is equivalent to   $L(t)=L_0(lnt)$ where $L_0(u)$ is convex on $(-\infty, +\infty )$.  We also assume that integrals involving $L(t)$ or its derivative are finite, and whenever necessary,	interchange of integral and derivative or limit  is permissible. 

For some $\lambda=\lambda_0$, we assume $\kappa(\lambda_0)=0$, and set $H(x;\lambda)=\int_{0}^{x}h(u;\lambda)I(u>\kappa(\lambda)) du$, $h(x)=h(x;\lambda_0)$, $H(x)=\int_{0}^{x}h(u) du$. When needed to stress the dependency of $\lambda=\lambda(\mu,\sigma,\tau)$ on $\tau$, we write $\lambda = \lambda(\tau)$ suppressing the other two parameters $\mu$ and $\sigma$. 

The function $\rho(\tau)$ is assumed to be monotone and without loss of generality we take it to be decreasing.  (In case it is increasing, we can take $\rho(1/\tau)$ as scaling factor in \eqref{1.1} replacing $G(\cdot)$ by the distribution function of $1/\tau$.)

Since  the (unconditional) density of $S/\sigma$,
\begin{equation}\label{pdfs1}
g_0(v)=\int_{0}^{\infty}(1/\rho(\tau)) g(v/\rho(\tau)) dG(\tau), 
\end{equation}
 is free of unknown parameters,  we may consider estimators of the form $\delta=cS$, $c>0$. The best choice $c_0$ that minimizes the risk with respect to $L(t)$, $EL(cS/\sigma)$,
 is the solution to the equation
\begin{equation}\label{bestconstant}
E\{L'(c_0S)S\}=\int_{0}^{\infty}L'(c_0v)vg_0(v) dv= 0,
\end{equation}
where the expectation is evaluated at $\sigma=1$ and the corresponding best constant multiple of $S$ is the estimator
\begin{equation}\label{bestestimator}
\delta_0=c_0S.
\end{equation}
Note that existence and uniqueness of $c_0$ is ensured by convexity.  For improving on the estimator $\delta_0$ 
we consider estimators of the form
\begin{equation}\label{phi}
\delta_{\varphi}= \begin{cases}
\varphi(W) S & \text{if}\ W>0,\\
\delta_0 & \text{otherwise},
\end{cases}
\end{equation}
where $W=X/S$ and $\varphi$  is an absolutely continuous and  positive function on $(0,\infty)$.  The following theorem gives sufficient conditions for the dominance of $\delta_{\varphi}$  over $\delta_0$. First, we need the following conditions.
\begin{description}
	\item (B1) $H(x;\lambda)/H(x)$ is nondecreasing in $x>0$ .
	\item (B2) $ H(x;\lambda(\tau))$ is nonincreasing in $\tau$ in the support of $G(\cdot)$.
	\end{description}
For later use we also state the following monotone likelihood ratio (mlr) conditions.

\begin{description}
	
	\item (A1) $g(c_1 x)/g(c_2 x)$ is increasing in $x>0$, for $0<c_1<c_2$.
	\item (A2) $h(c_1x)/h(c_2x)$ is  increasing in $x>0$, for $0<c_1<c_2$.
\end{description}

\begin{thm}\label{thm2.1}\textit{(the non-degenerate model)}
	Let $L'(t)$ or  $tL'(t)$ be increasing. Assume $(B1)$, $B(2)$ and the following conditions for $\varphi$:
	
	(a) $\varphi(w)$ is nondecreasing and $ \lim_{w\to\infty}\varphi(w)=c_0$.
	
	(b)$\int_{0}^{\infty}\rho(\tau)\int_{0}^{\infty}L'\big(\varphi(w)\rho(\tau)v\big)vg(v)H(wv) dv dG(\tau)\geq0$.
	
	Then $\delta_{\varphi}$ improves on $\delta_0$.
	\end{thm}
	\begin{proof}
	We employ \cite{Kubokawa1994}'s approach. Let $f(y, x;\lambda)$	be the joint (unconditional) density of $S/\sigma$ and $X/\sigma$, that is, $f(y,x;\lambda)=\int_{0}^{\infty}(1/\rho(\tau))^2 g(y/\rho(\tau))h( x/\rho(\tau);\allowbreak\lambda) I( x/\rho(\tau)>\kappa(\lambda)) dG(\tau)$. The difference in risks between $\delta_0$ and $\delta_{\varphi}$ is 
	  	\begin{align*}
	  RD&=EL(c_0S/\sigma)I(W>0) -EL(\varphi(W)S/\sigma)I(W>0) \\ &=	\int_{1}^{\infty}\frac{d}{dt}EL(\varphi(tW)S/\sigma)I(W>0) dt \\
	  &=	\int_{1}^{\infty}EL'(\varphi(tW)S/\sigma)(S/\sigma)W\varphi'(tW)I(W>0) dt\\
	  &=\int_{1}^{\infty}\int_{0}^{\infty}\int_{0}^{\infty}L'(\varphi(tw)y)y^2w\varphi'(tw)f(y,wy;\lambda) dw dy dt \\
	  &=\int_{0}^{\infty}\int_{0}^{\infty}\int_{w}^{\infty}L'(\varphi(u)y)y^2\varphi'(u)f(y,wy;\lambda) du dy dw \\
	  &=\int_{0}^{\infty}\varphi'(u)\int_{0}^{\infty}\int_{0}^{u}L'(\varphi(u)y)y^2f(y,wy;\lambda) dw dy du 
	  \end{align*}
	  \begin{equation}\label{general f}
	  =\int_{0}^{\infty}\varphi'(u)\int_{0}^{\infty}L'(\varphi(u)y)y\int_{0}^{uy}f(y,x;\lambda) dx dy du.
	  \end{equation}
	  We need to show that $RD\geq0$. Since, by condition $(a)$,  $\varphi'(u)\geq0$, it suffices to show
	  \begin{equation}\label{inequality}
	  A(u)=\int_{0}^{\infty}L'(\varphi(u)y)y\int_{0}^{uy}f(y,x;\lambda) dx dy \geq0,
	  \end{equation}
	  
	     for $u>0$. Substituting $f$ from above, we get
    \begin{align*}
    A(u)&=\int_{0}^{\infty}L'(\varphi(u)y)y\int_{0}^{uy}\int_{0}^{\infty}\big(1/\rho(\tau)\big)^2g( y/\rho(\tau))h(x/\rho(\tau);\lambda)\\
    &\qquad\qquad\qquad\times I( x/\rho(\tau)>\kappa(\lambda)) dG(\tau) dx dy \\
    &=\int_{0}^{\infty}L'(\varphi(u)y)y\int_{0}^{\infty}\big(1/\rho(\tau)\big) g( y/\rho(\tau))H( yu/\rho(\tau);\lambda) dG(\tau) dy\\
    &=\int_{0}^{\infty}g(v) \int_{0}^{\infty}\rho(\tau)v L'(\varphi(u)\rho(\tau)v) H(uv;\lambda(\tau)) dG(\tau) dv.
    \end{align*}
    Now let $tL'(t)$ be  increasing . Then $\rho(\tau)vL'(\varphi(u)\rho(\tau)v) $ is decreasing in $\tau$ (recall that $\rho(\tau)$ is decreasing)  and $H(uv;\lambda(\tau))$ is nonincreasing in $\tau$ by condition $(B2)$. Hence, by the covariance inequality (cf. \cite{LehmannCasella}, p.370) we get
    \begin{align*}
    A(u)&\geq\int_{0}^{\infty}g(v) \int_{0}^{\infty}\rho(\tau)vL'(\varphi(u)\rho(\tau)v ) dG(\tau) \int_{0}^{\infty}H(uv;\lambda(\tau)) dG(\tau) dv\\
    &=\int_{0}^{\infty}g(v)H(uv) \int_{0}^{\infty}\rho(\tau)vL'(\varphi(u)\rho(\tau)v)  dG(\tau) \int_{0}^{\infty}\frac{H(uv;\lambda(\tau))}{H(uv)} dG(\tau) dv \\
    &=\int_{0}^{\infty}g(v)H(uv)f_1(v)f_2(v) dv,
    \end{align*}
    where $f_1(v)=\int_{0}^{\infty}\rho(\tau)vL'(\varphi(u)\rho(\tau)v)  dG(\tau)$ and $f_2(v)=\int_{0}^{\infty}(H(uv;\lambda(\tau))/H(uv))\allowbreak dG(\tau)$. Then $f_1(v)$ is increasing since $tL'(t)$ is  increasing and $f_2(v)$ is nondecreasing by condition $(B1)$. Applying once again the covariance inequality, with respect to a density proportional to $g(v)H(uv)$,
    we get
    \begin{align*}
    A(u)&\geq\int_{0}^{\infty}g(v)H(uv)f_1(v) dv\int_{0}^{\infty}g(v)H(uv)f_2(v) dv\bigg/\int_{0}^{\infty}g(v)H(uv) dv.
    \end{align*}
    The last two integrals are positive. Thus, we will have $A(u)\geq0$ provided that the first one is nonnegative, that is, \\ $\int_{0}^{\infty}g(v)H(uv)f_1(v) dv =\int_{0}^{\infty}\rho(\tau)\int_{0}^{\infty}L'(\varphi(w)\rho(\tau)v)vg(v)H(wv) dv dG(\tau)\geq0$, which is condition $(b)$. So \eqref{inequality} holds true, and this completes the proof in the case $tL'(t)$ is increasing. 
    
    In the case that $L'(t)$ is increasing, the proof that $A(u)\geq0$ is similar and therefore it is omitted. Note that the covariance inequality will be applied with respect to the probability measure $\rho(\tau)dG(\tau)/\int_{0}^{\infty}\rho(\tau) dG(\tau)$,  and then with respect to $vg(v)H(uv)/\int_{0}^{\infty}vg(v)H(uv) dv$.
   	\end{proof}
	
	\begin{rem}\label{Remark2.1}
		In the non-degenerate case, whether the condition of  convexity  or convexity in log of the loss may be relaxed  to (strictly) bowl-shaped loss is an issue of interest.	In this respect, the part of the proof of Theorem \ref{thm2.1} up to 
		\eqref{inequality} holds for any joint density $f(y,x;\lambda)$ of $S/\sigma$ and $X/\sigma$. Thus, assuming that $L(t)$ is  bowl-shaped and $\int_{0}^{uy}f(y,x;\lambda) dx/\int_{0}^{uy}f(y,x; \lambda_0) dx$ is nondecreasing in $y$, $\delta_\varphi$ improves on $\delta_0$ provided conditions (a) and (b'')  $\int_{0}^{\infty}L'(\varphi(u)y)y\int_{0}^{uy}f(y,x;\lambda_0) dx dy\geq 0$ hold. For the scale mixtures framework, substituting $f$ as given in the beginning of the proof of Theorem \ref{thm2.1}, the last ratio of integrals becomes $\int_{0}^{\infty}(1/\rho(\tau)) g( y/\rho(\tau))H( yu/\rho(\tau);\lambda(\tau))\allowbreak dG(\tau)/ \int_{0}^{\infty}(1/\rho(\tau)) g( y/\rho(\tau))H( yu/\rho(\tau)) dG(\tau)$ and we need this to be  
		nondecreasing in $y$ for $\delta_{\varphi}$ to improve on $\delta_0$.
		  We have not pursued the monotonicity of this ratio any further.
	\end{rem}
	
	 When $G(\cdot)$ is degenerate at $\tau=1$, 	Theorem \ref{thm2.1} becomes \cite{Kubokawa1994}'s  Theorem 2.2, but with a stronger condition on $L(t)$.
	Condition (B2) is trivially satisfied (and in the proof of Theorem \ref{thm2.1}, the first application of the covariance inequality results in equality).  As is well known, the convexity condition is not really needed, and  may be relaxed to strictly bowl-shaped requiring,  though,  an mlr property for $g$, namely $(A1)$, to ensure existence and uniqueness of $c_{01}$ below. For the sake of uniformity with Theorem \ref{thm2.1}, we state the result for  the case $\tau=1$ wp1 also including the convexity conditions. Let $c_{01}$ be the solution to 
	\begin{equation}\label{c01}
	E\{L'(c_{01}S)S\}=0
	\end{equation}
	where the expectation is evaluated when $\tau=1$ wp1 and $\sigma=1$, and set
	\begin{equation}\label{delta01}
	\delta_{01}=c_{01}S,
	\end{equation}
	
	\begin{equation}\label{phi1}
	\delta_{1\varphi}= \begin{cases}
	\varphi_1(W)S & \text{if}\ W>0,\\
	\delta_{01} & \text{otherwise},
	\end{cases}
	\end{equation}
	where $\varphi_1$ is absolutely continuous and positive on ($0,\infty)$.
	\begin{thm}\label{degenerate}\textit{(the degenerate model, \cite{Kubokawa1994})} Let $\tau=1$ wp1. Let  $L(t)$ be  strictly bowl-shaped (which in particular holds when 
	$L'(t)$ or  $tL'(t)$ is increasing). Assume $(B1)$, $(A1)$ and the following conditions for $\varphi_1$:
		
		(a') $\varphi_1(w)$ is nondecreasing and $ \lim_{w\to\infty}\varphi_1(w)=c_{01}$.
		
		(b') $\int_{0}^{\infty}L'(\varphi_1(w)v)vg(v)H(wv) dv\geq0$.
		
		Then $\delta_{1\varphi}$ improves on $\delta_{01}$.
	\end{thm}

    We will show next that the construction of estimators  of $\sigma$  satisfying the conditions of Theorem \ref{thm2.1} reduces to the construction of estimators of $\sigma$   satisfying the conditions of Theorem \ref{degenerate}. To this end, we will employ Iliopoulos(1999)'s condition \eqref{Ilio}.
  
     Let $c_\tau$ be the unique solution to the equation
    \begin{equation}\label{best t}
    E\{L'(c_\tau \rho(\tau))\rho(\tau)\}=0, 
    \end{equation}

  In the theorem below, convexity of the loss or convexity in log is used only to ensure existence and uniqueness in \eqref{best t} and \eqref{c01}.
    \begin{thm}\label{thm2.3}
    Let $L'(t)$ or  $tL'(t)$ be increasing, and assume that  $L(t)$ satisfies \eqref{Ilio}.
     Assume that  $E\{\lambda_1(c_\tau \rho(\tau))\rho(\tau)\}$	and $E\{(\lambda_2(c_{01}S/(\rho(\tau)\sigma))\big(S/(\rho(\tau)\sigma)\big)\}$ are finite.
     Then it holds  $c_0=c_\tau c_{01}$ or, equivalently, $\delta_0=c_\tau \delta_{01}$.
    \end{thm}	
\begin{proof}
 Write $S/\sigma= \big(S/(\rho(\tau)\sigma)\big)\rho(\tau)$ and observe  that $\tau$ (and thus $\rho(\tau)$)	 and $ S/(\rho(\tau)\sigma)$ are independent, since given $\tau >0$ the density of $ S/(\rho(\tau)\sigma)$, $g(v)$,  does not involve $\tau $. Then Theorem \ref{thm2.3}  is merely Theorem 1 in  \cite{Bobotas2012} with $S_1=1/\rho(\tau)$ and $S_2 = S/(\rho(\tau)\sigma)$.
\end{proof}

Observe that $E\{(\lambda_2(c_{01}S/(\rho(\tau)\sigma))\big(S/(\rho(\tau)\sigma)\big)\}
=E\{\lambda_2(c_{01}S)S\}$, where the last expectation is evaluated when $\tau=1$ wp1 and $\sigma=1$. With regard to the   finiteness  of the expectations in Theorem \ref{thm2.3}, we refer to  Table 1 in \cite{Bobotas2012}, e.g., they hold for the entropy and Brown's losses.

 For our next result we need the following lemma. (Again  convexity does not play an essential role, the lemma holds with $L$ being strictly bowl-shaped and assuming existence - uniqueness of $c_\tau$, $c_{01}$.)
 
\begin{lem}\label{lemma} 
Assume the conditions of Theorem \ref{thm2.3}. Then $E\{\lambda_1(c_\tau\rho(\tau) )\rho(\tau)\}>0$.
\end{lem}
\begin{proof}
Setting, as in the proof of Theorem \ref{thm2.3}, $S_1=1/\rho(\tau)$ and $S_2 = S/(\rho(\tau)\sigma)$, Lemma \ref{lemma} is merely part of Lemma 1 in  \cite{Bobotas2012}.
\end{proof}	
\begin{thm}\label{reddegen}
Assume the conditions of Theorem \ref{thm2.3}. Also, assume (B1) and (B2). 
For $\delta_{1\varphi}$ as in \eqref{phi1},
 set $\delta=c_\tau \delta_{1\varphi}$,  so $\delta=\delta_\varphi$ as in \eqref{phi} with $\varphi(w)=c_\tau \varphi_1(w)$. Then $\delta$
	 improves on $\delta_0$, satisfying the conditions (a) and (b) of Theorem \ref{thm2.1} if and only if $\delta_{1\varphi}$ improves on $\delta_{01}$, satisfying the conditions (a') and (b') of Theorem \ref{degenerate}.

\end{thm}
\begin{proof}
We prove the result by showing that conditions (a), (b) of Theorem \ref{thm2.1} for $\varphi(w)$ are equivalent to conditions (a'), (b') of Theorem \ref{degenerate} for $\varphi_1(w)$. Since $\varphi(w)=c_\tau \varphi_1(w)$ and, by Theorem \ref{thm2.3}, $c_0=c_\tau c_{01}$, conditions (a) and (a')  are clearly equivalent. Denote by $D$, $D'$ the left-hand sides of conditions (b), (b'), respectively. Then, we get
\begin{align*}
D &= \int_{0}^{\infty}\rho(\tau)\int_{0}^{\infty}L'(c_\tau\rho(\tau)\varphi_1(w)v)vg(v)H(wv) dv dG(\tau)\\
  &=\int_{0}^{\infty}\rho(\tau)\int_{0}^{\infty}\{\lambda_1(c_\tau \rho(\tau))L'(\varphi_1(w)v)+
  \lambda_2(\varphi_1(w)v)L'(c_\tau\rho(\tau))\} \\
  &\qquad\qquad\qquad\times vg(v)H(wv) dv dG(\tau)\\ 
  &=D'\int_{0}^{\infty}\lambda_1(c_\tau \rho(\tau))\rho(\tau) dG(\tau)
  + \int_{0}^{\infty}L'(c_\tau \rho(\tau))\rho(\tau) dG(\tau)\int_{0}^{\infty}\lambda_2(\varphi_1(w)v)\\
   &\qquad\qquad\qquad\times vg(v)H(wv) dv \\
  &=D'E\{\lambda_1(c_\tau \rho(\tau))\rho(\tau)\},
\end{align*}
because of the definition of $c_\tau$ in \eqref{best t}. By Lemma \ref{lemma}, the last expectation  is positive, which entails that $D$ and $D'$ have the same sign. Thus, conditions (b), (b') are equivalent, and this completes the proof. 
\end{proof}
Let now $\delta_{1S}$ and $\delta_{1BZ}$ be the S-type and BZ-type estimators of $\sigma$ in the degenerate case $\tau =1$ wp1, which are available, e.g., from \cite{Kubokawa1994}. Then, by Theorem \ref{reddegen}, we directly obtain  the S-type and BZ-type estimators of $\sigma$ in the  non-degenerate case as $\delta_S=c_\tau \delta_{1S}$ and $\delta_{BZ}=c_\tau \delta_{1BZ}$. We note here that for $\delta_{1S}$ and $\delta_{1BZ}$ to satisfy conditions (a'), (b') we also need (as sufficient conditions) the mlr conditions $(A1)$ and $(A2)$.

\begin{rem}\label{Remark 2.4}\textbf{(robustness)}
	Under the conditions and assumptions of Theorem \ref{reddegen}, it can be shown that if  $\delta_{1\varphi}$ in \eqref{phi1}  satisfies conditions (a'), (b') of Theorem \ref{degenerate}, then it improves on $\delta_{01}$ not only when $G(\cdot)$ is degenerate at $\tau =1$, but also for any $G(\cdot)$ in the neighborhood  $\mathcal{F}= \{G(\cdot): c_\tau \leq 1 \}$. The degenerate $G(\cdot)$ belongs in $\mathcal{F}$ with $c_\tau=1$, which justifies the naming neighborhood. We sketch the proof. Following the arguments of  the proof of Theorem \ref{thm2.1} with $\varphi_1$ in the place of $\varphi$, it suffices to show (a'), which holds, and $\int_{0}^{\infty}\rho(\tau)\int_{0}^{\infty}L'(\varphi_1(w)\rho(\tau)v)vg(v)H(wv) dv dG(\tau) \geq 0$ for $G(\cdot) \in \mathcal{F}$. To this end, note that $\varphi(w)=c_\tau\varphi_1(w)$ satisfies condition (b) of Theorem \ref{thm2.1}, i.e., $\int_{0}^{\infty}\rho(\tau)\int_{0}^{\infty}L'(c_\tau\varphi_1(w)\rho(\tau)v)vg(v)H(wv) dv dG(\tau)\geq 0$ and use $c_\tau \leq 1$ and the convexity properties of $L(t)$. In a general decision theoretic context, the insensitivity of a dominance result to the assumed distribution is (misleadingly) termed as robustness of the dominating estimator (though it has nothing to do with influence curves or breakdown points).Thus, $\delta_{1\varphi}$ is a robust improved estimator. Similar robustness properties have been established by \cite{KuboSriva}, \cite{PetropoulosKourouklis2005}, and \cite{MaruyamaStrawderman2013}.
\end{rem}

 \begin{rem}\label{Remark2.3}
 	In the case of the entropy loss $L(t)=t-lnt-1$, Theorems \ref{thm2.1}, \ref{degenerate}, \ref{reddegen},  reobtain the results of  Theorems 2.1, 2.2, and 2.3(i) in \cite{PetropoulosKourouklis2005} where, however, the proofs (and a condition) strongly depend on the specific form the loss. Similar comment holds for the quadratic loss which is also treated in \cite{PetropoulosKourouklis2005}.  In this respect, the results of Section \ref{sect scale mixtures} unify and generalize  main findings in \cite{PetropoulosKourouklis2005}, while, also, explaining them.
 \end{rem}

\begin{exam}[\textbf{estimation of the scale parameter in scale mixtures of normal distributions}]
Let $\rho(\tau)$ be a positive and decreasing function. Given $\tau >0$, let $S$
and $T$ be independent statistics  such that $ S/\big(\rho(\tau)\sigma\big)$ and $T/\big(\rho(\tau)\sigma\big)^{1/2}$ have chi-square with $n$ degrees of freedom, $\chi_n^2$, and multivariate normal, $N_p(\mu/\big(\rho(\tau)\sigma\big)^{1/2}, I)$, distributions respectively. This is the  model of scale (variance) mixtures of normal distributions in  canonical form. See also \cite{PetropoulosKourouklis2005} for this model in a linear regression context. Set $X=||T||^2$, so that $ X/\big(\rho(\tau)\sigma\big)$ has a noncentral chi-square distribution with $p$ degrees of freedom and noncentrality parameter $||\mu||^2/\big(\rho(\tau)\sigma\big)$, $\chi_p^2(||\mu||^2/\big(\rho(\tau)\sigma\big))$. Thus, we obtain the model in \eqref{1.1}, with $g(\cdot) \sim \chi_n^2$, $h(\cdot;\lambda) \sim \chi_p^2(\lambda)$, $\lambda=||\mu||^2/\big(\rho(\tau)\sigma\big)$, and  $\kappa(\lambda)=0$.  Now, it is well known that $\chi_p^2(\lambda)$ has the mlr property with respect to $\lambda$, and hence it is also stochastically increasing in $\lambda$. Thus, condition (B1) holds, and since  $\rho(\tau)$ is decreasing, for $\tau_1<\tau_2$ we have $\lambda(\tau_1)<\lambda(\tau_2)$ which in turn implies $H(x;\lambda(\tau_1)>H(x;\lambda(\tau_2)$, i.e., condition (B2) holds too. Both conditions (A1) and (A2)  are the mlr property of a (central) chi-square distribution. The degenerate case $\tau=1$ wp1 corresponds to $S \sim \sigma \chi_n^2$  and $T \sim N_p(\mu, \sigma I)$, and S-type and  BZ-type estimators  of $\sigma$ are available in the literature, e.g., from \cite{Kubokawa1994}. Also \cite{Maruyama1998}-type (MA-type) estimators  for the quadratic and entropy losses  are known to satisfy the conditions (a'), (b') of Theorem \ref{degenerate}. Thus, in the non-degenerate case,  S-type, BZ-type, and MA-type estimators for $\sigma$ are readily derived from Theorem \ref{reddegen}, for each given $G(\cdot)$ and $L(t)$. 
 For instance, when $\rho(\tau)=1/\tau$ and  $G(\cdot) \sim \chi_\nu^2/\nu$ or $G(\cdot) \sim 1/\chi_{p+1}^2$ we get improved estimators of $\sigma$ in the case of a  multivariate $t$-distribution with $\nu$ degrees of freedom or a $p$-variate double exponential distribution, respectively. 
 For $\rho(\tau)=1/\tau$ and Brown's loss $ L(t)=(lnt)^2$, \eqref{best t} gives $lnc_\tau=Eln\tau$ and the neighborhood of robustness is $\mathcal{F}= \{G(\cdot): Eln\tau \leq0 \}$. Both the multivariate $t$-family and the $p$-variate double exponential family belong in $\mathcal{F}$, because $Eln\tau \leq lnE\tau=lnE(\chi_\nu^2/\nu)=0$ in the former case, and $Eln\tau=Eln(1/\chi_{p+1}^2)=-Eln(\chi_{p+1}^2)=-ln2-\Psi((p+1)/2)\leq -ln2-\Psi(1)<0$, for every $p\geq1$ in the latter ($\Psi$ is the digamma function). Thus, under Brown's loss the dominance of the  S-type and BZ-type estimators of the variance of a normal distribution extends over the families of $t$ and double exponential distributions.
  \end{exam}

\begin{exam}[\textbf{estimation of the scale parameter in scale mixtures of exponential distributions}]
For $\rho(\tau)$ a positive and decreasing function, let the data $\textbf{X}=(X_1,X_2,...,X_n), n\geq 2$, have density $f(x_1,x_2,...,x_1;\mu, \sigma)$ \\
=$\int_{0}^{\infty}\big(1/(\rho(\tau)\sigma\big)^n\big)exp\{(-1/\big(\rho(\tau)\sigma\big) \sum_{1}^{n}(x_i-\mu)\}I(x_{(1)}>\mu) dG(\tau)$, where $x_{(1)}=min\{x_1,x_2,...,x_n\}$, and $\mu \in (-\infty, +\infty)$,  $\sigma>0$ are unknown parameters. A minimal sufficient statistic for this data is $(S,X)$, 
where $S=\sum_{1}^{n}(X_i-X_{(1)})$ and $X=X_{(1)}$, so the model in \eqref{1.1} obtains with 
$g(\cdot)$ being the density of $\frac{1}{2}\chi_{2(n-1)}^2$, $h(x;\lambda)=nexp\{-n\big(x- \mu/\big(\rho(\tau)\sigma\big)\big)\}$, $\lambda= \mu/\big(\rho(\tau)\sigma\big)$ and $\kappa(\lambda)=\lambda$. Conditions $(A1)$, $(A2)$, and $(B1)$ are well known to hold, cf. \cite{Kubokawa1994}. Regarding $(B2)$, let $x>0$ and  observe that $H(x;\lambda(\tau))=\int_{0}^{x}nexp\{-n\big(u-\mu/\big(\rho(\tau)\sigma\big)\big)\}I(u>\mu/\big(\rho(\tau)\sigma\big)) du$. Thus, for $\mu=0$, we have $H(x;\lambda(\tau))=1- exp\{-nx\}$;  for $\mu>0$ and $x\leq \mu/\big(\rho(\tau)\sigma\big)$, we have $H(x;\lambda(\tau))=0$, whereas for $\mu>0$ and $x> \mu/\big(\rho(\tau)\sigma\big)$, we get $H(x;\lambda(\tau))=1-exp\{-nx\}exp\{n\mu/\big(\rho(\tau)\sigma\big)\}$; for $\mu <0$, $H(x;\lambda(\tau))=(1-exp\{-nx\})exp\{n\mu/\big(\rho(\tau)\sigma\big)\}$. In all cases, $H(x;\lambda(\tau))$ is nonincreasing in $\tau >0$ for $x>0$, and thus $(B2)$ holds. The degenerate case $\tau=1$ wp1 corresponds to an i.i.d. sample from the exponential distribution with parameters $(\mu, \sigma)$ for which S-type and  BZ-type estimators   of $\sigma$ are available in the literature, e.g., from \cite{Kubokawa1994}.  Also, MA-type estimators are given in  \cite{PetropoulosKourouklis2002} for the quadratic and entropy losses and are known to satisfy the conditions (a'), (b') of Theorem \ref{degenerate}.
 Thus, in the non-degenerate case,  from Theorem \ref{reddegen}, S-type,  BZ-type, and MA-type  estimators of $\sigma$ are directly derived, for each given $G(\cdot)$ and $L(t)$. For instance, when $\rho(\tau)=1/\tau$ and $G(\cdot)$ is a Gamma distribution with shape parameter $\alpha$ and scale parameter 1, or this Gamma divided by $\alpha$, we obtain improved estimators of the scale parameter of a multivariate Lomax or a modified multivariate Lomax distribution.
  In addition, for $L(t)=(lnt)^2$, these distributions belong in the neighborhood of robustness $\mathcal{F}=\{G(\cdot): Eln\tau\leq0\}$ for $0< \alpha <1$ and any $\alpha >0$, respectively, because in the former case, $Eln\tau= \Psi(\alpha) \leq \Psi(1) <0$ and in the latter, $Eln\tau \leq lnE\tau =Eln1=0$.
 \end{exam}

\section{Estimation of $ln\sigma$ and the entropy}\label{sect entropy}

In this section we consider the model in \eqref{1.1} and  aim at estimating the entropy of (the distribution of) $S$. Since $\sigma$ is a scale parameter for $S$, its entropy is given by $e(\sigma)=ln\sigma + e(1)$, where $e(1)$ is the entropy at $\sigma=1$, a known constant for given $\rho(\tau)$ and $G(\cdot)$. Specifically, it is  $e(1)=-\int_{0}^{\infty}lng_0(v)g_0(v) dv$ with $g_0(v)$ the density of $S/\sigma$ as in \eqref{pdfs1}.  Hence, it suffices to estimate $ln\sigma$.

 Let  $L_0(u)$ be a loss function on $(-\infty, +\infty)$ and set $L(t)=L_0(lnt)$. The losses $L_0(u)$ and  $L(t)$ will be used for the estimation of $ln\sigma$ and  $\sigma$, respectively. If  $L_0(u)$ is convex ($L_0'(u)$ increasing) with minimum 0 at $u=0$, then $L(t)$ is convex in log ($tL'(t)$ increasing) with minimum 0 at $t=1$. On the other hand, if $L_0(u)$ is strictly bowl-shaped  on $(-\infty, +\infty)$, i.e.,   decreasing for $u<0$ and increasing for $u>0$,  with minimum $0$ at $u=0$, then $L(t)$ is  strictly bowl-shaped on $(0, \infty)$, i.e., decreasing for $0<t<1$  and increasing for $t>1$,  with minimum $0$ at $t=1$.

 Based on $(S,X)$ we consider an estimator $\delta$ of $\sigma$  under $L(t)$ incurring loss $L(\delta/\sigma)$. Having $\delta $ at our disposal, we then take $\zeta=ln\delta$  as an estimator of $ln\sigma$ under $L_0(u)$ incurring loss $L_0(\zeta-ln\sigma)$.  The following proposition is elementary yet catalytic in estimating $ln\sigma$. Its proof does not require any condition on the losses, except for finite risks.
 \begin {prop}\label{prop 3.1}
 $\delta_1$ improves on   $\delta_2$ for estimating $\sigma$ under $L(t)=L_0(lnt)$ if and only if $\zeta_1=ln\delta_1$ improves on $\zeta_2=ln\delta_2$ for estimating $ln\sigma$
 under $L_0(u)$. 
 \end{prop}
 \begin{proof}
 For any $\delta$ and  $\zeta=ln\delta$  the two losses incurred are equal,  $L(\delta/\sigma)=L_0(ln(\delta/\sigma))\allowbreak= L_0(ln\delta-ln\sigma)=L_0(\zeta-ln\sigma)$.   
 \end{proof}
 
 By Proposition~\ref{prop 3.1}, estimation of $ln\sigma$ amounts to estimation of $\sigma$, and clearly is independent of the method used to estimate $\sigma$. Thus, from  the improved estimators of $\sigma$ obtained in Section \ref{sect scale mixtures}, we can directly derive improved estimators of $ln\sigma$ and $e(\sigma)$. We present them below.
 
 Let the loss $L_0(u)$ for estimating $ln\sigma$ be convex, so that  $L(t)=L_0(lnt)$ is  convex in log. Corresponding to an estimator $cS$ of $\sigma$, there is the estimator $lnc + lnS$ of $ln\sigma$ and the best one in terms of minimizing the risk, $EL_0(lnc+lnS-ln\sigma)$, is 
 \begin{equation}\label{zeta0}
 \zeta_0=ln\delta_0=lnc_0 +lnS,
 \end{equation}
  where $\delta_0=c_0S$ is as in \eqref{bestestimator} and $c_0$ as in \eqref{bestconstant}, satisfying $E\{L'(c_0S)S\}=\int_{0}^{\infty}L_0'(lnc_0 +lnv)g_0(v) dv= 0$. The following result provides sufficient conditions for improving on $\zeta_0$ and the entropy estimator $\zeta_0 + e(1)$. It builds  on Theorem \ref{thm2.1} and Proposition \ref{prop 3.1}.
  \begin{thm}\label{thm4.2}\textit{(the non-degenerate model)}
  		Let $L_0'(u)$  be increasing. Assume $(B1)$, $B(2)$ and the following conditions for $\delta_{\varphi}$ and $\varphi$ as in \eqref{phi}:
  	
  	(a) $\varphi(w)$ is nondecreasing and $ \lim_{w\to\infty}\varphi(w)=c_0$.
  	
  	(b)$\int_{0}^{\infty}\int_{0}^{\infty}L_0'\big(ln\varphi(w) + ln\rho(\tau) +lnv\big)g(v)H(wv) dv dG(\tau)\geq0$.
  	
  	Then, 	$\zeta_\varphi=ln\delta_\varphi$ and $\zeta_\varphi +e(1)$ improve on $\zeta_0$ and the entropy estimator $\zeta_0 + e(1)$, respectively.
  	\end{thm}
  \begin{proof}
  Set $L(t)=L_0(lnt)$, so $tL'(t)$ is increasing. Then, Theorem \ref{thm2.1} applies and $\delta_{\varphi}$ improves on $\delta_0$ for estimating $\sigma$ under $L(t)$. The result now follows from Proposition \ref{prop 3.1}.
  \end{proof}

In the case of the degenerate model, $\tau=1$ wp1 (in which case, recall, that $\rho(\tau)=1$ wp1), the convexity assumption of $L_0(u)$ is not needed and may be relaxed to strictly bowled-shaped. However, to ensure existence and uniqueness of $c_{01}$ in \eqref{c01} in the next result we invoke $(A1)$. With $\delta_{01}=c_{01}S$ as in \eqref{delta01}, the usual estimator of $ln\sigma$ in the degenerate model is
\begin{equation}
 \zeta_{01}=ln\delta_{01}=lnc_{01} +lnS.
 \end{equation}
 
   Note  that $\zeta_{01}$ also results from \eqref{zeta0} and that  the entropy becomes $e_1(\sigma)=ln\sigma +e_1(1)$, where $e_1(1)=-\int_{0}^{\infty}lng(v)g(v) dv$ is the entropy for $\tau=1$ wp1 and $\sigma=1$. On the other hand, if $L_0(u)$ is  convex the  result follows directly from Theorem \ref{thm4.2}.  
\begin{thm}\label{degentropy}\textit{(the degenerate model)} Let $\tau=1$ wp1. Also, let  $L_0(u)$ be strictly bowl-shaped. Assume $(B1)$, $(A1)$ and the following conditions for  $\delta_{1\varphi}$  and $\varphi_1$ as in \eqref{phi1}:
	
	(a') $\varphi_1(w)$ is nondecreasing and $ \lim_{w\to\infty}\varphi_1(w)=c_{01}$.
	
	(b') $\int_{0}^{\infty}L_0'(ln\varphi_1(w)+lnv)g(v)H(wv) dv\geq0$.
	
	Then $\zeta_{1\varphi}=ln\delta_{1\varphi}$ and $\zeta_{1\varphi}+e_1(1)$ improve on $\zeta_{01}$ and the entropy estimator $\zeta_{01} +e_1(1)$, respectively.
	
\end{thm}
\begin{proof}
	 Set $L(t)=L_0(lnt)$. Thus, $(a')$ and $(b')$ are \cite{Kubokawa1994}'s conditions (or the conditions of Theorem \ref{degenerate} in Section \ref{sect scale mixtures}) for  $\delta_{1\varphi}$ to improve on $\delta_{01}$ for estimating $\sigma$ under $L(t)$. The result now follows from Proposition \ref{prop 3.1}.
\end{proof}

 It is well known (\cite{Kubokawa1994}) that the S-type and BZ-type estimators of $\sigma$ satisfy $(a')$ and $(b')$ under $(B1)$, $(A1)$, $(A2)$. The latter conditions are also well known to hold for  the normal and  exponential distributions and after  suitable transformations for Pareto, Weibull, and Burr III.  Hence, for each of these models S-type and BZ-type estimators of $ln\sigma$ and the entropy $e_1(\sigma)=ln\sigma +e_1(1)$ are readily derived by Theorem~\ref{degentropy}, and are given below.

Let $\varphi_{2BZ}(w)$ and $\varphi_{2S}(w)$, respectively, be the solutions to the equations

\begin{equation}\label{phi1BZ}
 \int_{0}^{\infty}L_0'(ln\varphi_{2BZ}(w)+lnv)g(v)H(wv) dv = 0,
\end{equation}
\begin{equation}\label{phi2S}
 \int_{0}^{\infty}L_0'(ln\varphi_{2S}(w)+lnv)vg(v)h(wv) dv = 0,
\end{equation}
and  $\varphi_{3S}(w)=min\{c_{01},\varphi_{2S}(w)\}$. Set 
\begin{equation}\label{phisbz}
\delta_{1\varphi}^S = \varphi_{1S}(W)S, \,\,\,\delta_{1\varphi}^{BZ}=\varphi_{1BZ}(W)S
\end{equation}
 with $\varphi_{1S}(w)=\varphi_{3S}(w)$ and $\varphi_{1BZ}(w)=\varphi_{2BZ}(w)$ if $w>0$, whereas $\varphi_{1S}(w)=\varphi_{1BZ}(w)=c_{01}$ if $w\leq 0$, as the S-type and BZ-type estimators of $\sigma$ under $L(t)=L_0(lnt)$ when $\tau=1$ wp1. Then, the S-type and BZ-type estimators of $ln\sigma$ for the degenerate model under $L_0(u)$ are given by 
 \begin{equation}
 \zeta_{1\varphi}^S=ln\delta_{1\varphi}^S, \,\,\,\, \zeta_{1\varphi}^{BZ}=ln\delta_{1\varphi}^{BZ}.
 \end{equation}

Besides, $\zeta_{1\varphi}^S +e_1(1)$ and $\zeta_{1\varphi}^{BZ}+e_1(1)$ are the S-type and BZ-type estimators of the entropy $ln\sigma +e_1(1)$.

We return again to the non-degenerate model. If $L(t)=L_0(lnt)$, with $L_0(u)$  convex,  also satisfies \eqref{Ilio}, by Theorem \ref{reddegen} improved estimation of $\sigma$ under $L(t)$ reduces to the degenerate case and hence by Proposition \ref{prop 3.1} so does improved estimation of $ln\sigma$ and the entropy $ln\sigma +e(1)$ under $L_0(u)$. Also, Theorem \ref{thm2.3} applies so that  $c_0=c_\tau c_{01}$ and the usual estimator of $ln\sigma$ in \eqref{zeta0} is given by 
\begin{equation}\label{zeta00}
\zeta_0=lnc_0 +lnS = lnc_\tau + lnc_{01} +lnS = lnc_\tau + \zeta_{01}. 
\end{equation}

 Thus, the S-type and BZ-type estimators of $ln\sigma$ are expressed as 
  \begin{equation}\label{zetas}
 \zeta_\varphi^S=ln(c_\tau\delta_{1\varphi}^S) = lnc_\tau +ln\varphi_{1S}(W) +lnS =lnc_\tau + \zeta_{1\varphi}^S,
 \end{equation}
 \begin{equation}\label{zetabz} 
 \zeta_\varphi^{BZ}=ln(c_\tau\delta_{1\varphi}^{BZ}) = lnc_\tau + ln\varphi_{1BZ}(W) + lnS =lnc_\tau + \zeta_{1\varphi}^{BZ},
 \end{equation}
  and accordingly for the entropy $ln\sigma +e(1)$. 
 
  As expected, but remarkably, all the computations in \eqref{zeta00}, \eqref{zetas}, and \eqref{zetabz} are performed in the degenerate model and for the estimation of $\sigma$. Observe that $\zeta_0$, $\zeta_\varphi^S$, and  $\zeta_\varphi^{BZ}$ are linear translations (by the same displacement, $lnc_\tau$) of their degenerate model counterparts, $\zeta_{01}$, $\zeta_{1\varphi}^S$ and $\zeta_{1\varphi}^{BZ}$. When $\tau=1$ wp1, \eqref{best t} gives $c_\tau=1$, and  \eqref{zeta00}, \eqref{zetas}, and \eqref{zetabz} naturally  collapse to the estimators of the degenerate model. 
  
   Brown's loss function $L(t)=(lnt)^2$ and $L(t)=t-lnt-1$ (entropy loss) which correspond, respectively, to the convex losses $L_0(u)=u^2$ and $L_0(u)=e^u-u-1$ (Linex loss) are convex in log,  and satisfy \eqref{Ilio}. We refer to \cite{Bobotas2012} for some other convex in log losses satisfying \eqref{Ilio}. For all of them and for any scale mixture of normal or scale mixture of exponential distributions, the S-type and BZ-type estimators of $ln\sigma$ are given by \eqref{zetas} and \eqref{zetabz}. As an illustration,  we present below the case of scale mixtures of exponential distributions for the quadratic loss.

 \begin{exam}[\textbf{estimation of the entropy in scale mixtures of exponential distributions under the quadratic loss}]

 Let $L_0(u)=u^2$, corresponding to $L(t)=(lnt)^2$, and consider  any  scale mixture of exponential distributions with $\rho(\tau)=1/\tau$. As in Example 2.2, set  $S=\sum_{1}^{n}(X_i-X_{(1)})$,  $X=X_{(1)}$, and $W=X/S$.  From \eqref{c01}, \eqref{best t}, and \eqref{zeta00} we get $lnc_{01}=-ElnS=-\Psi(n-1)$, $lnc_\tau=Eln\tau$, and $\zeta_0= Eln\tau -\Psi(n-1) + lnS =Eln\tau + \zeta_{01}$. Also,  \eqref{phi1BZ} and \eqref{phi2S}  give $ln\varphi_{2BZ}(w)=-\Psi(n-1)-ln(1+nw)/\big((1+nw)^{n-1}-1\big)$ and $ln\varphi_{2S}(w)=-\Psi(n) +ln(1+nw)$. Consequently, from \eqref{phisbz}, for $w>0$, we have $ln\varphi_{1BZ}(w)=-\Psi(n-1)-ln(1+nw)/\big((1+nw)^{n-1}-1\big)$ and $ln\varphi_{1S}(w)=min\{-\Psi(n-1), -\Psi(n) + ln(1+nw)\} $ whereas, for $w\leq0$, $ln\varphi_{1BZ}(w)=ln\varphi_{1S}(w)=-\Psi(n-1)$. Substituting in \eqref{zetas} and \eqref{zetabz}, we obtain the S-type and BZ-type estimators of $ln\sigma$,
 $\zeta_\varphi^S$ and $\zeta_\varphi^{BZ}$, for the  non-degenerate model and, simultaneously,
 $\zeta_{1\varphi}^S$ and $\zeta_{1\varphi}^{BZ}$ for the degenerate model. The corresponding estimators of the entropy are $\zeta_\varphi^S +e(1)$ and $\zeta_\varphi^{BZ}+e(1)$ for the non-degenerate model, and  $\zeta_{1\varphi}^S +e_1(1)$ and $\zeta_{1\varphi}^{BZ} +e_1(1)$ for the degenerate model.
   The estimators  $\zeta_\varphi^S$ and $\zeta_\varphi^{BZ}$ are given by \cite{Petroetal}, whereas  $\zeta_{1\varphi}^S$ and $\zeta_{1\varphi}^{BZ}$ by \cite{Patra20}. Interestingly, equations \eqref{zetas} and \eqref{zetabz} help explaining each of the  components of the S-type and BZ-type estimators in the form obtained by \cite{Petroetal}. 
\end{exam}

 We anticipate  that  extension of entropy estimation in multivariate distributions is possible along the lines of the above results, but we have not pursued this endeavor. Scale mixtures of multivariate normal distributions (see \cite{PetropoulosKourouklis2005}) could be an interesting starting point. In the case of a multivariate normal $N_p(\mu, \Sigma)$, its entropy is a linear translation of  $ln|\Sigma|$ and improved estimators under the quadratic loss have been derived by \cite{Misraetal}. The same estimators can be obtained by using the improved estimators of $|\Sigma|$ under Brown's loss $L(t)=(lnt)^2$  given in \cite{IK} and then invoking Proposition \ref{prop 3.1}. 

%
%

\end{document}